\documentclass[11pt]{article}
\usepackage{graphicx}
\usepackage{amssymb}

\renewcommand{\theequation}{\arabic{equation}}
\newcommand{\D}{\displaystyle}

\newcommand{\p}{\partial}

\newcommand{\bxi}{\mbox{\boldmath$\xi$}}
\newcommand{\bx}{\mathbf{x}}
\newcommand{\bu}{\mathbf{u}}
\newcommand{\bn}{\mbox{\boldmath$\xi$}^\bot}

\newcommand{\myref}[1]{(\ref{#1})}

\newcommand{\gt}{\nabla ^\bot \theta}

\newenvironment{proof}{\noindent\textbf{Proof\ }}{\hspace*{\fill}$\Box$\medskip}

\newtheorem{theorem}{\textbf{Theorem}}[section]
\newtheorem{lemma}{\textbf{Lemma}}[section]
\newtheorem{remark}{\textbf{Remark}}[section]

\newtheorem{corollary}{\textbf{Corollary}}[section]
\begin{document}
\title{Dynamic Growth Estimates of Maximum Vorticity for 3D 
Incompressible Euler Equations and the SQG Model}
\author{Thomas Y. Hou\thanks{Applied and Comput. Math, Caltech,
Pasadena, CA 91125. {\it Email: hou@acm.caltech.edu.}} \and
Zuoqiang Shi\thanks{Applied and Comput. Math, Caltech,
Pasadena, CA 91125. {\it Email: shi@acm.caltech.edu.}} }
\date{\today}
\maketitle

\begin{abstract}
By performing estimates on the integral of the absolute value of 
vorticity along a local vortex line segment, we establish a relatively
sharp dynamic growth estimate of maximum vorticity under some 
assumptions on the local geometric regularity of the vorticity vector.
Our analysis applies to both the 3D incompressible Euler equations 
and the surface quasi-geostrophic model (SQG).
As an application of our vorticity growth estimate, we apply our
result to the 3D Euler equation with the two anti-parallel vortex tubes
initial data considered by Hou-Li \cite{HL06}. Under some 
additional assumption on the vorticity field, which seems to be
consistent with the computational results of \cite{HL06}, we show
that the maximum vorticity can not grow faster than double exponential
in time. Our analysis extends the earlier results by 
Cordoba-Fefferman \cite{CF01,CF02} and Deng-Hou-Yu \cite{DHY05,DHY06}.
\end{abstract}

\section{Introduction}

One of the most challenging problems in mathematical fluid
dynamics is to understand whether a solution of the 3D incompressible
Euler equations can develop a finite time singularity from
smooth initial data with finite energy. A main difficulty
is due to the presence of the vortex stretching term, which
has a formal quadratic nonlinearity in vorticity.
This problem has attracted a lot of attention in the
mathematics community and many people have contributed to
its understanding, see the recent book by Majda and
Bertozzi \cite{MB02} for a review of this subject.

An important development in recent years is the work
by Constantin, Fefferman, and Majda who showed that the local
geometric regularity of vortex lines can lead to depletion of
 nonlinear vortex stretching \cite{CFM96}. Inspired by
the work of \cite{CFM96}, Deng, Hou, and Yu \cite{DHY05,DHY06}
obtained more localized non-blowup criteria by exploiting the 
geometric regularity of a vortex line segment whose arclength
may shrink to zero at the potential singularity time. To obtain 
these results, Deng-Hou-Yu \cite{DHY05,DHY06} used a Lagrangian 
approach and explored the connection between the local geometric 
regularity of vortex lines and the growth of vorticity. Guided 
by this local geometric non-blowup analysis, Hou and Li
\cite{HL06,HL08} performed large scale computations with resolution
up to $1536\times 1024\times 3072$ to re-examine some of
the most well-known blow-up scenarios, including the two slightly
perturbed anti-parallel vortex tubes that was originally investigated
by Kerr \cite{Kerr93}. The computations of Hou and Li \cite{HL06}
provide strong numerical evidence that the geometric regularity
of vortex lines, even in an extremely localized region near
the support of maximum vorticity, can lead to depletion of
vortex stretching. We refer to a recent survey paper 
\cite{Hou09} for more discussions on this topic.

In this paper, we derive new growth rate estimates of
maximum vorticity for the 3D incompressible Euler equations.
We use a framework similar to that adopted by 
Deng-Hou-Yu \cite{DHY05}. The main innovation of this work 
is to introduce a method of analysis to study the dynamic
evolution of the integral of the absolute value of vorticity 
along a local vortex line segment. Specifically, we
derive a dynamic estimate for the quantity:
\begin{eqnarray}
Q(t)=\frac{1}{L(t)}\int_{0}^{L(t)}|\omega(\bx(s,t),t)| ds ,
\end{eqnarray}
where $\bx(s,t)$ is a parameterization of a vortex line segment,
$L^t$, and $L(t)$ is the arclength of $L^t$. The assumption
on $\bx(s,t)$ is less restrictive than that in \cite{DHY05}.
As in \cite{DHY05}, we assume that the vorticity along
$L^t$ is comparable to the maximum vorticity, i.e.
$\max_{L^t} |\omega| \ge c_0 \|\omega\|_{L^\infty}$.
Let $V(t)=\max_{\bx\in L^t}|(\bu\cdot \bxi)(\bx,t)|$,
and $U(t)=\max_{\bx\in L^t}|(\bu\cdot \bn)(\bx,t)|$.
Here $\bxi$ be the unit vorticity vector of $L^t$, and
$\bn$ the unit normal vector. Under the assumption that 
$\int_{L^t} |\bxi \cdot \nabla \bxi | ds \leq C_0$
and $\int_{L^t} |\nabla \cdot \bxi | ds \leq C_0$, 
we derive a relatively sharp growth estimate for $Q(t)$,
which can be used to obtain an upper bound on the
growth rate of the maximum vorticity: 
\begin{eqnarray}
\|\omega(t)\|_{L^\infty}
\le \frac{Q(T_0)}{c_0}\exp\left(C_0+\int_{T_0}^{t}
\frac{C_VV(t')+C_UU(t')}{L(t')}dt'\right),
\label{Vorticity_growth}
\end{eqnarray}
where $C_U$ and $C_V$ depend on $C_0$.
If we further assume that $L(t)$ has a positive lower bound, 
the above estimate implies no blow-up up to $t=T$, if 
$\int_0^T \|\mathbf{u}\|_{\infty} dt < \infty$.
This generalizes the result of Cordoba and Fefferman \cite{CF01}.

The above estimate extends the result of Deng-Hou-Yu in \cite{DHY05}.
In fact, it is easy to check that under the assumption that
$U(t)+V(t) \le C_u (T-t)^{-A}$ and $L(t) \ge C_L (T-t)^B$ with
$A+B < 1$, the right hand side of \myref{Vorticity_growth} remains 
bounded up to the time $t=T$, implying no blow-up up to $t=T$.
Our result can be also applied to the critical case when $A+B = 1$, 
which was considered in \cite{DHY06}. 
In this case, we have
  \begin{eqnarray}
\D\frac{C_VV(t)+C_UU(t)}{L(t)}\sim \frac{1}{T-t}.
  \end{eqnarray}
If we further assume that there exists $C_w <1$ such that
\begin{eqnarray}
\D\frac{C_VV(t)+C_UU(t)}{L(t)}\le \frac{C_w}{T-t},
  \end{eqnarray}
where $C_w$ depends on $C_0$, and the scaling constants in
$U(t)$, $V(t)$ and $L(t)$, then our growth estimate implies that 
\begin{equation}
\|\omega(t) \|_{L^\infty} \le  \frac{C}{(T-t)^{C_w}}.
\label{Vorticity_critical}
\end{equation}
Application of the Beale-Kato-Majda non-blow-up criterion
\cite{BKM84} would exclude blow-up at $t=T$ since $C_w < 1$ 
implies $\int_0^T \|\omega (t)\|_{L^\infty} dt < \infty$.

Of particular interest is the case when the vorticity has a
local Clebsch representation. In this case, the vorticity can be
represented by the two Clebsch variables $\phi$ and $\psi$
near the support of maximum vorticity as follows:
\begin{equation}
\omega = \nabla \phi \times \nabla \psi,
\end{equation}
where $\phi$ and $\psi$ are carried by the flow, that is
\begin{eqnarray}
&&\phi_t + {\bf u} \cdot \nabla \phi = 0, \\
&&\psi_t + {\bf u} \cdot \nabla \psi = 0, 
\end{eqnarray}
where ${\bf u}$ is the velocity field. In addition to the
geometric regularity assumption on $L^t$,
if we further assume that one of the Clebsch
variables has a bounded gradient and $L(t) \ge L_0 >0$,
then we prove that
the maximum vorticity can not grow faster than double
exponential in time, i.e. 
$\|\omega (t) \|_{L^\infty} \le C \exp(\exp(c_0t))$.
 
As an application of this result, we re-examine the computations 
of the 3D incompressible Euler equations with the
two slightly perturbed anti-parallel vortex tubes initial data
by Hou and Li \cite{HL06}. By examining the vorticity field 
carefully near the support of maximum vorticity (see Fig. 1), 
the vorticity field seems to have a local Clebsch representation. 
One of the Clebsch variables may be chosen along the vortex 
tube direction, which appears to be regular.  Moreover, 
the vortex lines within the support of maximum vorticity seem
to be quite smooth and has length of order one, implying that 
$L(t)$ has a positive lower bound. Thus the result that we described above 
may apply. One of the important findings of the Hou-Li computations 
is that the maximum vorticity does not grow faster than double 
exponential in time. Our new estimate on the vorticity growth 
may offer a theoretical explanation to the mechanism that
leads to this dynamic depletion of vortex stretching.

 \begin{figure}
    \begin{center}
      \includegraphics[width=8cm]{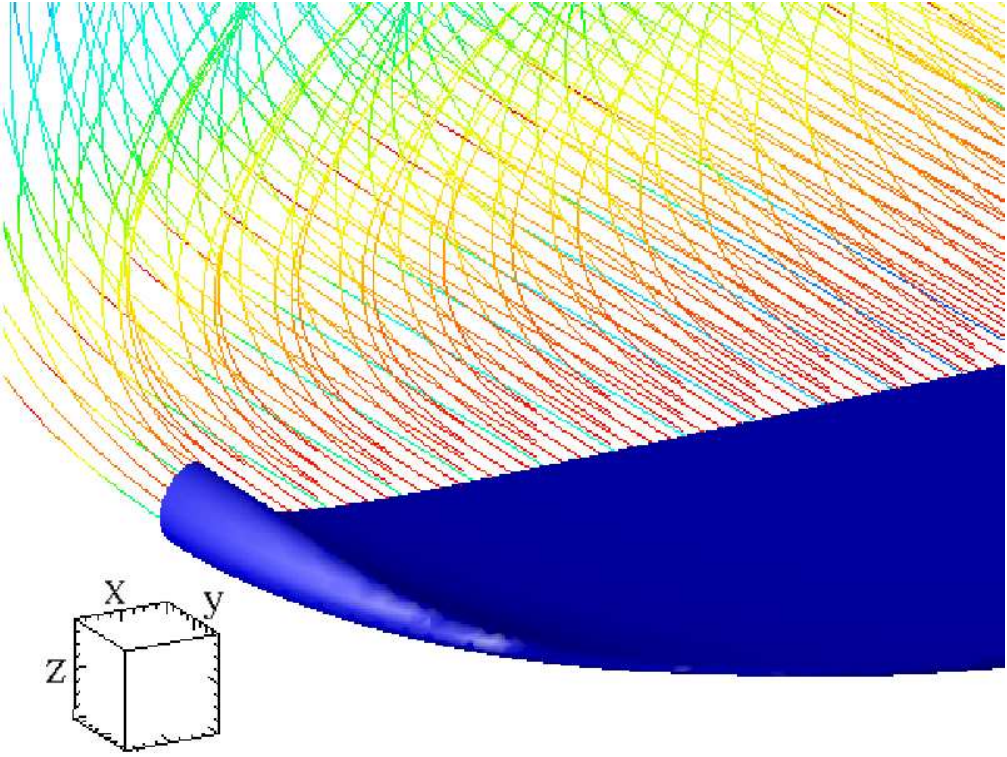}
    \end{center}
    \caption{The local 3D vortex structures and vortex lines around the
      maximum vorticity at $t=17$. Computation from Hou and Li
\cite{HL06} for the 3D incompressible Euler equations with
two slightly perturbed anti-parallel vortex tubes initial data.}
\end{figure}

We also apply our method of analysis to the surface 
quasi-geostrophic model (SQG) \cite{CMT94}. As pointed out in
\cite{CMT94}, a formal analogy between the SQG model and the 3D Euler 
equations can be established by considering $\nabla ^\bot \theta$ 
as the corresponding vorticity in the 3D Euler equations. Here 
$\theta$ is a scalar quantity that is transported by the flow:
\begin{equation}
\theta_t + {\bf u} \cdot \nabla \theta = 0,
\quad {\bf u} = \nabla^\bot (-\Delta )^{-1/2} \theta. 
\end{equation}
Let $L^t$ be a level set segment of $\theta$ along which 
$|\gt|$ is comparable to $\|\gt\|_{L^\infty}$
 and denote by $\bxi$ the unit tangent vector of $L^t$.
Under the assumption that
$\int_{L^t} |\bxi \cdot \nabla \bxi | ds \leq C_0$
and $\int_{L^t} |\nabla \cdot \bxi | ds \leq C_0$, 
we obtain a much better growth estimate for
 $\|\nabla^\bot \theta\|_{L^\infty}$:
\begin{eqnarray}
   \|\nabla^\bot \theta (t)\|_{L^\infty}
\le C \exp\left(c_1\exp\left(\int_{T_0}^t
\frac{c_2}{L(t')}dt'\right)\right).
\end{eqnarray}
In particular, if $L(t) \ge L_0 >0$, the above estimate
implies that 
$\|\nabla^\bot\theta\|_{L^\infty} \le C \exp(\exp(c_0 t))$.
This seems to be consistent with the numerical results 
obtained in \cite{OY97,CNS98}, see also \cite{CF02,DHLY05}.

The rest of the paper is organized as follows. In Section 2, we 
derive our estimate on the integral of vorticity over a vortex 
line segment for the 3D Euler equations, and apply this estimate 
to obtain an upper bound for the dynamic growth rate of maximum 
vorticity.  In Section 3, we generalize our analysis to the SQG 
model. In the Appendix, we prove a technical result  for the 3D 
Euler equations which states that the maximum velocity is 
bounded by $C \log (\|\omega(t)\|_{L^\infty})$ when the vorticity
field has a local Clebsch representation and one of the Clebsch
variables has a bounded gradient.

\section{Vorticity growth estimate for the 3D Euler equations}

In this section, we derive a new dynamic growth estimate of 
the maximum vorticity for the 3D incompressible Euler equations.
We adopt a framework similar to that used in \cite{DHY05}.
Let $\Omega(t) = \|\omega (t) \|_{L^\infty}$. We consider, 
at time $t$, a vortex line segmant $L^t$ along which the 
maximum of $|\omega| $ (denoted by $\Omega_L(t)$ in the
following) is comparable to $\Omega(t)$.  We use 
$\bx(s,t), \; 0\le s\le L(t)$ to parameterize $L^t$ with 
$s$ being the arclength variable. In our paper, we do not 
assume that $L^t$ is a subset of $\mathbf{X}(L^{t'},t',t)$, 
the flow image of $L^{t'}$ at time $t$, for $t'<t$. This 
assumption was required in the analysis of \cite{DHY05}.
Further, we denote by $L(t)$ the arclength of $L^t$. 
The unit tangential and normal vectors are defined as
follows:$\bxi=\frac{\nabla ^\bot \theta}{|\nabla ^\bot \theta|},\;
\bn=\frac{\bxi\cdot \nabla \bxi}{|\bxi\cdot \nabla \bxi|}$,
the unsigned curvature is defined as
$\kappa=|\bxi\cdot \nabla \bxi|$, and 
$\tau=\nabla \cdot \bxi$. Finally, 
we denote $V(t)=\max_{\bx\in L^t}|(\bu\cdot \bxi)(\bx,t)|$, 
and $U(t)=\max_{\bx\in L^t}|(\bu\cdot \bn)(\bx,t)|$. 

\begin{lemma}
\label{lemma euler}
  Let  $L^t=\left\{\bx(s,t), 0\le s\le L(t)\right\}$ be a family of 
vortex line segments which come from the same vortex line. Define 
$Q(t)$ as the mean of $|\omega(\bx,t)|$ over $L^t$,
\begin{eqnarray}
Q(t)=\frac{1}{L(t)}\int_{0}^{L(t)}|\omega(\bx(s,t),t)| ds .
\end{eqnarray}
Then, we have 
\begin{eqnarray}
\frac{dQ(t)}{dt}&=&\frac{1}{L}\left(\int_{0}^{L(t)}2\tau|\omega|(\bu\cdot\bxi)
 ds-\int_{0}^{L(t)}\kappa|\omega|(\bu\cdot\bn) ds\right.\nonumber\\
&&\left.-\int_{0}^{L(t)}\kappa\left(|\omega|(\bx(s,t),t)-|\omega|(\bx(L(t),t),t)\right)
(\bu\cdot\bn) ds\right)\nonumber\\
&&+\frac{1}{L}\left[\left(\bu\cdot\bxi\right)(\bx(L,t),t)-\left(\bu\cdot\bxi\right)(\bx(0,t),t)\right]|\omega|(\bx(L,t),t)\nonumber\\
&&+\frac{1}{L}\left[|\omega(\bx(L,t),t)|-|\omega(\bx(0,t),t)|\right]\left(\frac{d \bx}{dt}\cdot \bxi+\bu\cdot \bxi\right)(\bx(0,t),t)\nonumber\\
&&+\frac{L_t}{L}\left(|\omega(\bx(L,t),t)|-Q\right). \quad\quad
\end{eqnarray}
\end{lemma}

\begin{proof}
  Differentiating $Q(t)$ with respect to $t$ yields
\begin{eqnarray}
\label{dQw}
\frac{dQ(t)}{dt}&=&\frac{1}{L(t)}\frac{d}{dt}\left(\int_{0}^{L(t)}|\omega(\bx(s,t),t)| ds\right)-\frac{QL_t}{L} .
\end{eqnarray}

Let $\beta$ be the arclength parameter of this vortex line at time $T_0$. Then we can write, for this specific vortex line,
$s=s(\beta,t)$. Let $\beta_1(t), \beta_2(t)$ be the corresponding 
coordinates of the end points of $L^t$, i.e.
\begin{eqnarray*}
s(\beta_1(t),t)=0,\quad s(\beta_2(t),t)=L(t).
\end{eqnarray*}

First, we can change the integral variable from $s$ to $\beta$ in 
\myref{dQw},
\begin{eqnarray}
\label{dQw-1}
\frac{d}{dt}\left(\int_{0}^{L(t)}|\omega(\bx(s,t),t)| ds\right)=
\frac{d}{dt}\left(\int_{\beta_1(t)}^{\beta_2(t)}|\omega(\bx(s(\beta,t),t),t)|
s_\beta d\beta\right) .
\end{eqnarray}
In \cite{DHY05}, Deng-Hou-Yu proved the following equality,
\begin{eqnarray}
\frac{d s}{d\beta}(\bx(\beta,t),t)=
\frac{|\omega(\bx(\beta,t),t)|}{|\omega(\bx(\beta,T_0),T_0)|} .
\end{eqnarray}
Substituting the above relation to \myref{dQw-1} yields
\begin{eqnarray}
&&\frac{d}{dt}\left(\int_{0}^{L(t)}|\omega(\bx(s,t),t)| ds\right)\nonumber\\
&=& \frac{d}{dt}\left(\int_{\beta_1(t)}^{\beta_2(t)}|\omega(\bx(\beta,t),t)|^2/|\omega(\bx(\beta,T_0),T_0)| d\beta\right)\nonumber\\
&=&\int_{\beta_1(t)}^{\beta_2(t)}\frac{2|\omega|}{|\omega(\bx(\beta,T_0),T_0)|}\frac{D}{Dt}|\omega|
 d\beta+\frac{|\omega(\bx(\beta_2,t),t)|^2}{|\omega(\bx(\beta_2,T_0),T_0)|}\beta_{2t}
-\frac{|\omega(\bx(\beta_1,t),t)|^2}{|\omega(\bx(\beta_1,T_0),T_0)|}\beta_{1t}\nonumber\\
&=&\int_{\beta_1(t)}^{\beta_2(t)}\frac{2|\omega|}{|\omega(\bx(\beta,T_0),T_0)|}\alpha|\omega|
 d\beta+\frac{|\omega(\bx(\beta_2,t),t)|^2}{|\omega(\bx(\beta_2,T_0),T_0)|}\beta_{2t}
-\frac{|\omega(\bx(\beta_1,t),t)|^2}{|\omega(\bx(\beta_1,T_0),T_0)|}\beta_{1t}\nonumber\\
&=&\int_{0}^{L(t)}2\alpha|\omega|
 ds+\frac{|\omega(\bx(\beta_2,t),t)|^2}{|\omega(\bx(\beta_2,T_0),T_0)|}\beta_{2t}
-\frac{|\omega(\bx(\beta_1,t),t)|^2}{|\omega(\bx(\beta_1,T_0),T_0)|}\beta_{1t} ,
\end{eqnarray}
where we have used
$\frac{D}{Dt}|\omega| = \alpha |\omega|$ with
$\alpha = (\bxi \cdot \nabla) \bu \cdot \bxi = (\bu\cdot\bxi)_s-\kappa \bu\cdot \bn$ \cite{DHY05}.

Note that the arclength $L(t)$ can be expressed as follows:
\begin{eqnarray}
L(t)=\int_{\beta_1(t)}^{\beta_2(t)}s_\beta d\beta=\int_{\beta_1(t)}^{\beta_2(t)}\frac{|\omega(\bx(\beta,t),t)|}{|\omega(\bx(\beta,T_0),T_0)|} d\beta .
\end{eqnarray}
Differentiating the both sides with respect to $t$, we get
\begin{eqnarray}
\label{dbeta2}
\frac{dL(t)}{dt}&=&\int_{\beta_1(t)}^{\beta_2(t)}\frac{D|\omega|/Dt}{|\omega(\bx(\beta,T_0),T_0)|} d\beta+
\frac{|\omega(\bx(\beta_2,t),t)|\beta_{2t}}{|\omega(\bx(\beta_2,T_0),T_0)|}-
\frac{|\omega(\bx(\beta_1,t),t)|\beta_{1t}}{|\omega(\bx(\beta_1,T_0),T_0)|}\nonumber\\
&=&\int_{\beta_1(t)}^{\beta_2(t)}\frac{\alpha |\omega|}{|\omega(\bx(\beta,T_0),T_0)|} d\beta+
\frac{|\omega(\bx(\beta_2,t),t)|\beta_{2t}}{|\omega(\bx(\beta_2,T_0),T_0)|}-
\frac{|\omega(\bx(\beta_1,t),t)|\beta_{1t}}{|\omega(\bx(\beta_1,T_0),T_0)|}\nonumber\\
&=&\int_{0}^{L(t)}\alpha ds+
\frac{|\omega(\bx(\beta_2,t),t)|\beta_{2t}}{|\omega(\bx(\beta_2,T_0),T_0)|}-
\frac{|\omega(\bx(\beta_1,t),t)|\beta_{1t}}{|\omega(\bx(\beta_1,T_0),T_0)|} .
\end{eqnarray}
Substituting the above relation to \myref{dQw-1}, we obtain
\begin{eqnarray}
\label{dQw-1-update}
&&\frac{d}{dt}\left(\int_{0}^{L(t)}|\omega(\bx(s,t),t)| ds\right)\nonumber\\
&=&\int_{0}^{L(t)}2\alpha|\omega|
 ds+|\omega(\bx(\beta_2,t),t)|\left(L_t-\int_{0}^{L(t)}\alpha ds\right)\nonumber\\
&&+\left(|\omega(\bx(\beta_2,t),t)|-|\omega(\bx(\beta_1,t),t)|\right)
\frac{|\omega(\bx(\beta_1,t),t)|}{|\omega(\bx(\beta_1,T_0),T_0)|}\beta_{1t} .
\end{eqnarray}
Observe that
\begin{eqnarray}
\frac{d \bx(0,t)}{dt}\cdot \bxi(\bx(0,t),t)
&=&\frac{d \bx(\beta_1(t),t)}{dt}\cdot \bxi(\bx(\beta_1(t),t),t)\nonumber\\
&=& \left(\frac{\p \bx(\beta_1,t)}{\p t}+\frac{\p \bx(\beta_1,t)}{\p \beta_1}\frac{d \beta_1}{d t}\right)
\cdot \bxi(\bx(\beta_1(t),t),t)\nonumber\\
&=&\left(\bu(\beta_1,t)+\frac{\omega(\bx(\beta_1,t),t)\beta_{1t}}{|\omega(\bx(\beta_1,T_0),T_0)|}\right)
\cdot \bxi(\bx(\beta_1(t),t),t)\nonumber\\
&=&\left(\bu\cdot \bxi\right)(\bx(0,t),t)+\frac{|\omega(\bx(\beta_1,t),t)|\beta_{1t}}
{|\omega(\bx(\beta_1,T_0),T_0)|} .
\end{eqnarray}
Substituting the above equality to \myref{dQw-1-update}, we get
\begin{eqnarray}
\label{dQw-1-update2}
&&\frac{d}{dt}\left(\int_{0}^{L(t)}|\omega(\bx(s,t),t)| ds\right)\nonumber\\
&=&\int_{0}^{L(t)}2\alpha|\omega|
 ds+|\omega(\bx(L,t),t)|\left(L_t-\int_{0}^{L(t)}\alpha ds
\right)\nonumber\\
&&+\left(|\omega(\bx(L,t),t)|-|\omega(\bx(0,t),t)|\right)\left(\frac{d \bx}{dt}\cdot \bxi-\bu\cdot \bxi
\right)(\bx(0,t),t) .
\end{eqnarray}
Now, we have
\begin{eqnarray}
\label{dQw-equal}
\frac{dQ(t)}{dt}&=&\frac{1}{L(t)}\frac{d}{dt}\left(\int_{0}^{L(t)}|\omega(\bx(s,t),t)| ds\right)-\frac{QL_t}{L}\nonumber\\
&=&\frac{1}{L(t)}\left(\int_{0}^{L(t)}2\alpha|\omega|
 ds-|\omega(\bx(\beta_2,t),t)|\int_{0}^{L(t)}\alpha ds\right)\nonumber\\
&&+\frac{1}{L(t)}\left(|\omega(\bx(\beta_2,t),t)|-|\omega(\bx(\beta_1,t),t)|\right)\left(\frac{d \bx}{dt}\cdot \bxi-\bu\cdot \bxi
\right)(\bx(0,t),t)\nonumber\\
&&+\frac{L_t}{L}\left(|\omega(\bx(\beta_2,t),t)|-Q\right).
\end{eqnarray}
Using $\alpha=(\bu\cdot\bxi)_s-\kappa \bu\cdot \bn$ and 
integrating by parts, we obtain
\begin{eqnarray}
&& \int_{0}^{L(t)}2\alpha|\omega| ds-|\omega(\bx(\beta_2,t),t)|\int_{0}^{L(t)}\alpha ds\nonumber\\
&=&\int_{0}^{L(t)}2|\omega|\left((\bu\cdot\bxi)_s-\kappa \bu\cdot \bn\right) ds
-|\omega(\bx(L,t),t)|\int_{0}^{L(t)}\left((\bu\cdot\bxi)_s-\kappa \bu\cdot \bn\right) ds\nonumber\\
&=&2\left(|\omega|\bu\cdot\bxi\right)\left|_{0}^{L(t)}\right.+\int_{0}^{L(t)}2\tau|\omega|(\bu\cdot\bxi) ds
-\int_{0}^{L(t)}2\kappa|\omega|\bu\cdot \bn ds\nonumber\\
&&-|\omega|(\bx(L,t),t)\left(\bu\cdot\bxi\right)\left|_{0}^{L(t)}\right.+
|\omega(\bx(L,t),t)|\int_{0}^{L(t)}\kappa \bu\cdot \bn ds .
\end{eqnarray}
Substitute the above equality to \myref{dQw-equal} gives
\begin{eqnarray}
\frac{dQ(t)}{dt}&=&\frac{1}{L}\left(\int_{0}^{L(t)}2\tau|\omega|(\bu\cdot\bxi)
 ds-\int_{0}^{L(t)}\kappa|\omega|(\bu\cdot\bn) ds\right.\nonumber\\
&&\left.-\int_{0}^{L(t)}\kappa\left(|\omega|(\bx(s,t),t)-|\omega|(\bx(L(t),t),t)\right)
(\bu\cdot\bn) ds\right)\nonumber\\
&&+\frac{1}{L}\left[\left(\bu\cdot\bxi\right)(\bx(L,t),t)-\left(\bu\cdot\bxi\right)(\bx(0,t),t)\right]|\omega|(\bx(L,t),t)\nonumber\\
&&+\frac{1}{L}\left[|\omega(\bx(L,t),t)|-|\omega(\bx(0,t),t)|\right]\left(\frac{d \bx}{dt}\cdot \bxi+\bu\cdot \bxi\right)(\bx(0,t),t)\nonumber\\
&&+\frac{L_t}{L}\left(|\omega(\bx(L,t),t)|-Q\right). \quad\quad
\end{eqnarray}
This completes the proof of Lemma \ref{lemma euler}. 
\end{proof}

Now we are ready to state the main result of this paper.

\begin{theorem}
\label{theorem 3d euler}
Assume there is a family of vortex line segments $L^t=\left\{\bx(s,t), 0\le s\le L(t)\right\}$ which come from the same
 vortex line and $T_0\in [0,T)$, such that $\Omega_L(t)\ge c_0\Omega(t)$ for some $0<c_0\le 1$ for all $t\in [T_0,T)$ and
$|\omega(\bx(L(t),t),t)|=\Omega_L(t)$. Further
we assume that there exist constants $C_0>0,\;C_l>0$, such that the
following condition is satisfied:
\begin{eqnarray}
\D\int_{L^t} |\kappa(\bx(s,t),t)|ds\le C_0,\quad 
\int_{L^t} |\tau(\bx(s,t),t)|ds \le C_0,
\quad \left|\frac{d \bx(0,t)}{dt}\cdot \bxi(\bx(0,t),t)\right|\le C_lV(t).
\end{eqnarray}
Then, the maximum vorticity $\Omega(t)$ satisfies the following growth 
estimate:
\begin{eqnarray}
\Omega(t)\le \frac{Q(T_0)}{c_0}\exp\left(C_0+\int_{T_0}^{t}\frac{C_VV(t')+C_UU(t')}{L(t')}dt'\right),
\end{eqnarray}
where $C_U=C_0(2C_1-1)$, $C_V=2C_0C_1+(C_1-1)(C_l+1)+2C_1$ and $C_1=\exp(C_0)$.
\end{theorem}
\begin{proof}
Without the loss of generality, we may assume that $L(t)$ is monotonically decreasing, i.e. $L_t\le 0$ and $L(T_0)$ 
is sufficiently small.

In Lemma 1 of \cite{DHY05}, Deng-Hou-Yu proved the following equality:
\begin{eqnarray}
  |\omega(\bx(s_2,t))|= |\omega(\bx(s_1,t))|\;e^{\int_{s_1}^{s_2}-\tau(\bx(s,t))ds}.
\end{eqnarray}
It follows from the assumption $\int_{L^t} |\tau(\bx,t)|ds \le C_0$ that 
\begin{eqnarray}
\label{max-Qw}
 \max_{\bx\in L^t} |\omega(\bx,t)|\le C_1 \min_{\bx\in L^t} |\omega(\bx,t)|\le C_1 Q,
\end{eqnarray}
where $C_1=\exp(C_0)$.

By Lemma \ref{lemma euler}, we have
\begin{eqnarray}
\label{dQw-update}
\frac{dQ(t)}{dt}&=&\frac{1}{L}\left(\int_{0}^{L(t)}2\tau|\omega|(\bu\cdot\bxi)
 ds-\int_{0}^{L(t)}\kappa|\omega|(\bu\cdot\bn) ds\right.\nonumber\\
&&\left.-\int_{0}^{L(t)}\kappa\left(|\omega|(\bx(s,t),t)-|\omega|(\bx(L(t),t),t)\right)
(\bu\cdot\bn) ds\right)\nonumber\\
&&+\frac{1}{L}\left[\left(\bu\cdot\bxi\right)(\bx(L,t),t)-\left(\bu\cdot\bxi\right)(\bx(0,t),t)\right]|\omega|(\bx(L,t),t)\nonumber\\
&&+\frac{1}{L}\left[|\omega(\bx(L,t),t)|-|\omega(\bx(0,t),t)|\right]\left(\frac{d \bx}{dt}\cdot \bxi+\bu\cdot \bxi\right)(\bx(0,t),t)\nonumber\\
&&+\frac{L_t}{L}\left(|\omega(\bx(L,t),t)|-Q\right)\nonumber\\
&\equiv & I_1+I_2+I_3+I_4.
\end{eqnarray}

Recall that we choose the endpoint $\bx(L,t)$ of $L^t$ such that
$|\omega(\bx(L,t),t)|=\Omega_L$ which implies that 
$|\omega(\bx(L,t),t)|\ge Q$. We also have $L_t\le 0$ by our assumption.
Thus, we conclude that
\begin{eqnarray}
\label{dQw-2}
I_4 = \frac{L_t}{L}\left(|\omega(\bx(L,t),t)|-Q\right) \le 0.
\end{eqnarray}
To estimate $I_3$, we use the
assumption
$\left|\frac{d \bx(0,t)}{dt}\cdot \bxi(\bx(0,t),t)\right|\le C_lV(t)$,
which implies that
\begin{eqnarray}
\label{dQw-1-3}
I_3&=&\frac{1}{L}\left(|\omega(\bx(\beta_2,t),t)|-|\omega(\bx(\beta_1,t),t)|\right)
\left|\left(\frac{d \bx}{dt}\cdot \bxi+\bu\cdot \bxi\right)(\bx(0,t),t)\right|\nonumber\\
&\le& (C_1-1)(C_l+1)VQ/L .
\end{eqnarray}
It remains to estimate $I_1$ and $I_2$ on the right hand side of
\myref{dQw-update}. First of all, $I_1$ can be estimated as follows:
\begin{eqnarray}
\label{dQw-1-1}
I_1&=&\frac{1}{L}\left(\int_{0}^{L(t)}2\tau|\omega|(\bu\cdot\bxi)
 ds-\int_{0}^{L(t)}\kappa|\omega|(\bu\cdot\bn) ds\right.\nonumber\\
&&\left.-\int_{0}^{L(t)}\kappa\left(|\omega|(\bx(s,t),t)-|\omega|(\bx(L(t),t),t)\right)
(\bu\cdot\bn) ds\right)\nonumber\\
&\le&\left( 2C_0C_1VQ+C_0C_1UQ+C_0(C_1-1)UQ\right)/L.
\end{eqnarray}
As for $I_2$, we proceed as follows::
\begin{eqnarray}
\label{dQw-1-2}
I_2=\frac{1}{L}\left[\left(\bu\cdot\bxi\right)(\bx(L,t),t)-\left(\bu\cdot\bxi\right)(\bx(0,t),t)\right]|\omega|(\bx(L,t),t)
\le2C_1VQ/L .
\end{eqnarray}
Now, combining \myref{dQw-2}, \myref{dQw-1-3}, \myref{dQw-1-1} and \myref{dQw-1-2}, we obtain the following estimate for $\D\frac{dQ(t)}{dt}$,
\begin{eqnarray}
\label{dQ-final}
\frac{dQ(t)}{dt}&\le& \frac{Q}{L}\left(C_0(2C_1-1) U+(2C_0C_1+(C_1-1)(C_l+1)+2C_1)V\right)\nonumber\\
&=&\frac{Q}{L}\left(C_U U +C_V V\right),
\end{eqnarray}
where $C_U=C_0(2C_1-1)$, $C_V=2C_0C_1+(C_1-1)(C_l+1)+2C_1$.
It follows from the above inequality that
\begin{eqnarray}
Q(t)\le Q(T_0)\exp\left(\int_{T_0}^{t}\frac{C_VV(t')+C_UU(t')}{L(t')}dt'\right).
\end{eqnarray}
Therefore, we have proved that
\begin{eqnarray}
\Omega(t)\le \frac{\Omega_L(t)}{c_0}\le \frac{C_1}{c_0}Q(t)\le \frac{Q(T_0)}{c_0}
\exp\left(C_0+\int_{T_0}^{t}\frac{C_VV(t')+C_UU(t')}{L(t')}dt'\right).
\end{eqnarray}
This completes the proof of Theorem \ref{theorem 3d euler}.
\end{proof}

\begin{remark}
  If we further assume $L(t)$ has a positive lower bound, then the 
above growth estimate for the maximum vorticity implies no blowup up 
to $t=T$, if $\|\mathbf{u}\|_{\infty}$ is integrable from 0 to $T$. 
This extends the result of Cordoba and Fefferman \cite{CF01}.

\end{remark}

\begin{corollary}
  In the critical case when
$\D\frac{C_VV(t)+C_UU(t)}{L(t)}\sim \frac{1}{T-t}$,   
if we further assume that there exists a positive constant
$C_w<1$ such that
  \begin{eqnarray}
\label{assumption-critical}
\D\frac{C_VV(t)+C_UU(t)}{L(t)}\le \frac{C_w}{T-t},   
  \end{eqnarray}
then the solution remains regular up to time $T$.
\label{Euler-corollary}
\end{corollary}
\begin{proof}
  Using Theorem \ref{theorem 3d euler} and the assumption \myref{assumption-critical}, we have
  \begin{eqnarray}
    \int_{T_0}^{T}\Omega(t)dt&\le&  \int_{T_0}^{T} \frac{Q(T_0)}{c_0}
\exp\left(C_0+\int_{T_0}^{t}\frac{C_VV(t')+C_UU(t')}{L(t')}dt'\right) dt\nonumber\\
&\le& \frac{Q(T_0)}{c_0}\exp(C_0) \int_{T_0}^{T}
\exp\left(\int_{T_0}^{t}\frac{C_w}{T-t'}dt'\right) dt\nonumber\\
&=&  \frac{Q(T_0)}{c_0}\exp(C_0)(T-T_0)^{C_w} 
\int_{T_0}^{T}\frac{dt}{(T-t)^{C_w}}
< +\infty ,
  \end{eqnarray}
since $0<C_w < 1$.
Then, the Beale-Kato-Majda non-blowup criterion \cite{BKM84} implies 
that there is no blowup up to time $T$.
\end{proof}

\begin{remark}
We remark that Corollary \ref{Euler-corollary} 
generalizes the result of Deng-Hou-Yu in \cite{DHY06} with less
restrictive requirement on the scaling constants. More specifically,
  if there is $A\in [0,1]$ and positive constants $C_v,\;C_0,\;C_L$, such that
  \begin{eqnarray*}
    V(t),\;U(t)&\le& C_v(T-t)^{-A},\\
\int_{L^t}|\kappa|ds,\; \int_{L^t} |\tau|ds &\le& C_0,\\ 
L(t) &\ge& C_L (T-t)^{1-A},
  \end{eqnarray*}
then Corollary \ref{Euler-corollary} implies that
there is no blowup up to time $T$, as long as the following condition 
is satisfied:
\begin{eqnarray}
  C_v(C_U+C_V)<C_L.
\end{eqnarray}
\end{remark}

\begin{theorem}
\label{theorem-euler-L0}
Suppose that all the assumptions in Theorem \ref{theorem 3d euler} hold.
If we further assume 
\begin{eqnarray}
\max_{\bx\in L^t}|\mathbf{u}(\bx,t)|\le C_u\log \Omega(t),
\end{eqnarray}
then the maximum vorticity is bounded by the following growth estimate:
\begin{eqnarray}
   \Omega(t)\le \exp\left(\log \left(\frac{C_1}{c_0}Q(T_0)\right)\exp\left(\int_{T_0}^t\frac{C}{L(t')}dt'\right)\right),
\end{eqnarray}
where $C=C_u\max(C_U,C_V)$.
\end{theorem}
\begin{proof} 
The assumption $\max_{\bx\in L^t}|\mathbf{u}(\bx,t)|\le C_u\log \Omega(t)$ implies that
\begin{eqnarray}
  U,\;V \le C_u\log \Omega(t)\le C_u\log\left(\frac{\Omega_L(t)}{c_0}\right) \le  C_u\log\left(\frac{C_1}{c_0}Q\right)=
 C_u\left(\log Q+\log\left(\frac{C_1}{c_0}\right)\right) .
\end{eqnarray}
Substituting the above inequality to \myref{dQ-final} in the proof of 
Theorem \ref{theorem 3d euler}, we obtain
\begin{eqnarray}
\label{dQ-L0}
\frac{dQ(t)}{dt}\le \frac{Q}{L}(C_UU+C_VV)\le \frac{C}{L}Q\left(\log Q+\log\left(\frac{C_1}{c_0}\right)\right),
\end{eqnarray}
where $C=C_u\max(C_U,C_V)$.
Solving the above differential inequality gives 
\begin{eqnarray}
Q(t)\le\frac{c_0}{C_1}\exp\left(\left(\log Q(T_0)+\log \left(\frac{C_1}{c_0}\right)\right)
\exp\left(\int_{T_0}^t\frac{C}{L(t')}dt'\right)\right),
\end{eqnarray}
which immediately yields the desired growth estimate for $\Omega (t)$:
\begin{eqnarray}
  \Omega(t)\le \frac{C_1}{c_0}Q\le \exp\left(\log \left(\frac{C_1}{c_0}Q(T_0)\right)
\exp\left(\int_{T_0}^t\frac{C}{L(t')}dt'\right)\right).
\end{eqnarray}
This completes the proof of Theorem \ref{theorem-euler-L0}.
\end{proof}

\begin{remark}
The assumption $\max_{\bx\in L^t}|\mathbf{u}(\bx,t)|\le C_u\log \Omega(t)$ may
appear strong. We remark that under certain assumption on the
local vorticity structure around the vortex line segments $L^t$, 
this property can be justified. Specifically, suppose that
the vorticity field admits a Clebsch representation in a region
$\Omega_0(t)\subset \mathbb{R}^3$ with diameter $O(1)$ containing $L^t$.
This implies that there exist two level set functions 
$\phi,\;\psi : \Omega_0(t) \rightarrow \mathbb{R}$ 
such that the vorticity can be represented as follows:
\begin{eqnarray}
  \omega=\left(\nabla \phi \times \nabla \psi\right), \quad
\bx \in \Omega_0(t) ,
\end{eqnarray}
where $\phi$ and $\psi$ are carried by the flow, that is
\begin{eqnarray}
  \phi_t+\bu\cdot \nabla \phi=0,\\
 \psi_t+\bu\cdot \nabla \psi=0,
\end{eqnarray}
with smooth initial data that decay rapidly at infinity. 
If we further assume that one of the level set functions has
a bounded gradient and there exists a small constant $\rho>0$ such that
$ \bigcup_{\bx\in L^t} B(\bx,\rho)\subset \Omega_0(t)$, 
where $B(\bx,\rho)$ is a ball whose center is $\bx$ and radius 
is $\rho$, then we can show that the maximum velocity over $L^t$ satisfies
\begin{eqnarray}
  \max_{\bx\in L^t}|\mathbf{u}(\bx,t)|\le C_u\log \Omega(t).
\end{eqnarray}
The proof of this results will be given in the Appendix.
\end{remark}

One immediate consequence of Theorem \ref{theorem-euler-L0} is the
following Corollary.
\begin{corollary}
  If in the statement of Theorem \ref{theorem-euler-L0} we further assume that 
  \begin{eqnarray}
    L(t)\ge L_0>0,
  \end{eqnarray}
then $\Omega(t)$ can not grow faster than double exponential in time.
\end{corollary}

\section{Growth estimates for the SQG model}

In this section, we will apply the method of analysis presented
in the previous section to study the dynamic growth of 
$\|\nabla^\bot \theta\|_{L^\infty}$ 
for the SQG model. First, we state an estimate 
for the maximum velocity obtained by D. Cordoba in \cite{Cor98}.
\begin{lemma}
\label{vel-QG}
For the SQG model, there exists a generic constant $C_u>0$
such that for $t>0$,
\begin{eqnarray}
\|{\bf u}(\cdot,t)\|_{L^\infty}\le C_u \log \|\nabla ^\bot \theta\|_{L^\infty}.
\end{eqnarray}
\end{lemma}

Let $\Omega(t) = \|\nabla ^\bot \theta\|_{L^\infty}$.
We consider, at time $t$, a level set segmant $L^t$
along which the maximum of $|\nabla ^\bot \theta|$ 
(denoted by $\Omega_L(t)$ in the following) is comparable to
 $\Omega(t)$. We use the same notations as in the previous
section. First, we prove the corresponding estimate for 
$Q(t)$ for the SQG model.

\begin{lemma}
\label{lemma SQG}
  Let  $L^t=\left\{\bx(s,t), 0\le s\le L(t)\right\}$ be a family of level set segments which come from the same
 level set, and $Q(t)$ be the average of $|\gt|$ over $L^t$,
\begin{eqnarray}
Q(t)=\frac{1}{L(t)}\int_{0}^{L(t)}|\gt(\bx(s,t),t)| ds .
\end{eqnarray}
Then, we have
\begin{eqnarray}
\frac{dQ(t)}{dt}&=&\frac{1}{L}\left(\int_{0}^{L(t)}2\tau|\gt|(\bu\cdot\bxi)
 ds-\int_{0}^{L(t)}\kappa|\gt|(\bu\cdot\bn) ds\right.\nonumber\\
&&\left.-\int_{0}^{L(t)}\kappa\left(|\gt|(\bx(s,t),t)-|\gt|(\bx(L(t),t),t)\right)
(\bu\cdot\bn) ds\right)\nonumber\\
&&+\frac{1}{L}\left[\left(\bu\cdot\bxi\right)(\bx(L,t),t)-\left(\bu\cdot\bxi\right)(\bx(0,t),t)\right]|\gt|(\bx(L,t),t)\nonumber\\
&&+\frac{1}{L}\left[|\gt(\bx(L,t),t)|-|\gt(\bx(0,t),t)|\right]\left(\frac{d \bx}{dt}\cdot \bxi+\bu\cdot \bxi\right)(\bx(0,t),t)\nonumber\\
&&+\frac{L_t}{L}\left(|\gt(\bx(L,t),t)|-Q\right). \quad\quad
\end{eqnarray}
\end{lemma}
\begin{proof}
  The proof follows exactly the same procedure as in the proof of 
Lemma \ref{lemma euler} in the previous section by using the equality
\begin{eqnarray}
\frac{d s}{d\beta}(\bx(\beta,t),t)=
\frac{|\gt(\bx(\beta,t),t)|}{|\gt(\bx(\beta,T_0),T_0)|} 
\end{eqnarray}
which holds for the SQG model, see \cite{DHLY05}. We will not reproduce
the proof here.
\end{proof}

By following the same procedure as in the proof of Theorem 
\ref{theorem-euler-L0}, we can obtain the following 
growth estimate for the SQG model:
\begin{theorem}
\label{theorem-SQG}
Assume there is a family of level set segments $L^t=\left\{\bx(s,t), 0\le s\le L(t)\right\}$ and $T_0\ge 0$, 
 such that $\Omega_L(t)\ge c_0\Omega(t)$ for some $0<c_0\le 1$ 
and $|\omega(\bx(L(t),t),t)|=\Omega_L(t)$ for $t \ge T_0$. 
Further, we assume that there exist constants $C_0>0,\; C_l>0$ such that
\begin{eqnarray}
\D\int_{L^t} |\kappa(\bx(s,t),t)|ds\le C_0,\quad 
\int_{L^t} |\tau(\bx(s,t),t)|ds \le C_0,
\quad \left|\frac{d \bx(0,t)}{dt}\cdot \bxi(\bx(0,t),t)\right|\le C_l V(t),
\end{eqnarray}
for $t \geq T_0$.
Then, the maximum of $|\nabla^\bot \theta |$ is bounded by the following
estimate: 
\begin{eqnarray}
   \Omega(t)\le \exp\left(\log \left(\frac{C_1}{c_0}Q(T_0)\right)\exp\left(\int_{T_0}^t\frac{C}{L(t')}dt'\right)\right),
\end{eqnarray}
where $C=C_u\max(C_U,C_V)$, $C_u$ is the constant given in Lemma \ref{vel-QG}, $C_1=\exp(C_0)$, $C_U,\; C_V$ are same as 
those defined in Theorem \ref{theorem 3d euler}.
\end{theorem}

\begin{corollary}
In addition to the assumptions stated in Theorem \ref{theorem-SQG}, 
if we further assume that $L(t)$ has a positive lower bound, 
i.e. $L(t)\ge L_0>0$, then $\Omega(t)$ does not grow faster than double
exponential in time. More precisely, we have
\begin{eqnarray}
 \Omega(t)\le \exp\left(\log \left(\frac{C_1}{c_0}Q(T_0)\right)\exp\left(\frac{C}{L_0} (t-T_0)\right)\right).
\end{eqnarray}
\end{corollary}

  \renewcommand{\theequation}{A-\arabic{equation}}
  \setcounter{equation}{0}  
  \section*{Appendix}

\
\begin{lemma}
Assume that $\omega(\bx,t)$ has a local Clebsch representation 
in a region $\Omega_0(t)\subset \mathbb{R}^3$ containing $L^t$,
 i.e. there exist two level
set functions,
$\phi,\;\psi: \Omega_0(t) \rightarrow \mathbb{R}$
such that the vorticity can be expressed as follows:
\begin{eqnarray}
  \omega=\left(\nabla \phi \times \nabla \psi\right), \quad \bx \in 
\Omega_0(t) ,
\end{eqnarray}
where $\phi$ and $\psi$ are carried by the flow, that is
\begin{eqnarray}
  \phi_t+\bu\cdot \nabla \phi=0,\\
 \psi_t+\bu\cdot \nabla \psi=0,
\end{eqnarray}
with smooth initial data that decay rapidly at infinity.
If one of the level set functions has a bounded gradient, and there 
exists a small constant $\rho>0$ such that
$ \bigcup_{\bx\in L^t} B(\bx,\rho)\subset \Omega_0(t)$, where $B(\bx,\rho)$ is a ball whose center is $\bx$ and radius 
is $\rho$,
then the maximum velocity over $L^t$ satisfies the following estimate:
\begin{eqnarray}
 \max_{\bx\in L^t}|\mathbf{u}(\bx,t)|\le C_u\log \Omega(t). 
\end{eqnarray}
\end{lemma}
\begin{proof}
Without the loss of generality, we may assume that
$|\nabla \psi | \le C$.
By the well-known Biot-Savart Law \cite{MB02}, we have
  \begin{eqnarray}
    \bu(\bx,t)&=&\frac{1}{4\pi}\int_{\mathbb{R}^3} \frac{\mathbf{y}}{|\mathbf{y}|^3}\times \omega(\bx+\mathbf{y},t)d\mathbf{y}
,\quad \forall \bx\in L^t.
  \end{eqnarray}
Define a smooth cut-off function 
$\chi: \{0\}\cup \mathbb{R}^+ \rightarrow [0,1]$, such that 
$\chi(r)=1$ for $0 \le r\le 1$ and $\chi(r)=0$ for $r\ge 2$. 
Let $0<\delta<\rho/2$ be a small positive parameter to be determined later. 
Then we have
 \begin{eqnarray}
   |\bu(\bx,t)|&=&\frac{1}{4\pi}\left|\int_{\mathbb{R}^3} \frac{\mathbf{y}}{|\mathbf{y}|^3}\times \omega
(\bx+\mathbf{y},t)d\mathbf{y}\right|,\nonumber\\
&\le&\frac{1}{4\pi}\left|\int_{\mathbb{R}^3}\chi\left(\frac{|\mathbf{y}|}{\delta}\right) \frac{\mathbf{y}}{|\mathbf{y}|^3}\times \omega
(\bx+\mathbf{y},t)d\mathbf{y}\right|\nonumber\\
&&\quad\quad +\frac{1}{4\pi}\left| \int_{\mathbb{R}^3}\left(1-\chi\left(\frac{|\mathbf{y}|}{\delta}\right)\right) \frac{\mathbf{y}}{|\mathbf{y}|^3}\times \omega
(\bx+\mathbf{y},t)d\mathbf{y}\right|\nonumber\\
&\equiv & I_1+I_2.
 \end{eqnarray}
By a direct calculation, we get
\begin{eqnarray}
\label{I1}
  I_1\le C\delta\; \Omega .
\end{eqnarray}
To estimate $I_2$, we split it into two terms as follows: 
\begin{eqnarray}
  I_2&\le& \frac{1}{4\pi}\left| \int_{\mathbb{R}^3}\chi\left(\frac{|\mathbf{y}|}{\rho}\right)\left(1-\chi\left(\frac{|\mathbf{y}|}{\delta}\right)\right) \frac{\mathbf{y}}{|\mathbf{y}|^3}\times \omega
(\bx+\mathbf{y},t)d\mathbf{y}\right|\nonumber\\
&&+\frac{1}{4\pi}\left| \int_{\mathbb{R}^3}\left(1-\chi\left(\frac{|\mathbf{y}|}{\rho}\right)\right)\left(1-\chi\left(\frac{|\mathbf{y}|}{\delta}\right)\right) \frac{\mathbf{y}}{|\mathbf{y}|^3}\times \omega
(\bx+\mathbf{y},t)d\mathbf{y}\right|\nonumber\\
&=&\frac{1}{4\pi}\left| \int_{\mathbb{R}^3}\chi\left(\frac{|\mathbf{y}|}{\rho}\right)\left(1-\chi\left(\frac{|\mathbf{y}|}{\delta}\right)\right) \frac{\mathbf{y}}{|\mathbf{y}|^3}\times \omega
(\bx+\mathbf{y},t)d\mathbf{y}\right|\nonumber\\
&&+\frac{1}{4\pi}\left| \int_{\mathbb{R}^3}\left(1-\chi\left(\frac{|\mathbf{y}|}{\rho}\right)\right) \frac{\mathbf{y}}{|\mathbf{y}|^3}\times \omega
(\bx+\mathbf{y},t)d\mathbf{y}\right|\nonumber\\
&\equiv &I_3+I_4.
\end{eqnarray}

We first estimate $I_4$. Integration by parts gives
\begin{eqnarray}
  I_4&=&\frac{1}{4\pi}\left| \int_{\mathbb{R}^3}\left(1-\chi\left(\frac{|\mathbf{y}|}{\rho}\right)\right)
 \frac{\mathbf{y}}{|\mathbf{y}|^3}\times\left( \nabla\times \bu\right)(\bx+\mathbf{y},t)d\mathbf{y}\right|\nonumber\\
&= &\frac{1}{4\pi}\left| \int_{\mathbb{R}^3}\left(\bu(\bx+\mathbf{y},t)\times\nabla\right)\times\left[\left(1-\chi\left(\frac{|\mathbf{y}|}{\rho}\right)\right)
 \frac{\mathbf{y}}{|\mathbf{y}|^3}\right]d\mathbf{y}\right|\nonumber\\
&\le &\frac{1}{4\pi}\left| \int_{\mathbb{R}^3}\left[\bu(\bx+\mathbf{y},t)\times
\nabla\left(1-\chi\left(\frac{|\mathbf{y}|}{\rho}\right)\right)\right]
 \frac{\mathbf{y}}{|\mathbf{y}|^3}d\mathbf{y}\right|\nonumber\\
&&+\frac{1}{4\pi}\left| \int_{\mathbb{R}^3}\left(1-\chi\left(\frac{|\mathbf{y}|}{\rho}\right)\right)
\left[\left(\bu(\bx+\mathbf{y},t)\times\nabla\right)\times
 \frac{\mathbf{y}}{|\mathbf{y}|^3}\right]d\mathbf{y}\right|\nonumber\\
&\equiv& A+B.
\end{eqnarray}

By a direct calculation and using the  H\"older inequality, we can estimate
each term defined in the above expression as follows:
\begin{eqnarray}
\label{I4-1}
  A&\le& C\rho^{-3/2}\|\bu\|_2,\\
\label{I4-2}
B&\le & C\rho^{-3/2}\|\bu\|_2.
\end{eqnarray}

To estimate $I_3$, using the assumptions 
$\omega=\left(\nabla \phi \times \nabla \psi\right)= \nabla\times (\phi\nabla\psi),\; \forall \bx \in \Omega_0(t)$ and 
$ \bigcup_{\bx\in L^t} B(\bx,\rho)\subset \Omega_0(t)$, we can 
to split $I_3$ into three terms for any $\bx\in L^t$:
\begin{eqnarray}
I_3  &=&
\frac{1}{4\pi}\left| \int_{\mathbb{R}^3}\chi\left(\frac{|\mathbf{y}|}{\rho}\right)\left(1-\chi\left(\frac{|\mathbf{y}|}{\delta}\right)\right) \frac{\mathbf{y}}{|\mathbf{y}|^3}\times\omega(\bx+\mathbf{y},t)d\mathbf{y}\right|
\nonumber\\
 &=&
\frac{1}{4\pi}\left| \int_{\mathbb{R}^3}\chi\left(\frac{|\mathbf{y}|}{\rho}\right)\left(1-\chi\left(\frac{|\mathbf{y}|}{\delta}\right)\right) \frac{\mathbf{y}}{|\mathbf{y}|^3}\times\left(\nabla\times \left(\phi\nabla \psi\right)\right)(\bx+\mathbf{y},t)d\mathbf{y}\right|
\nonumber\\
 &=& \frac{1}{4\pi}\left| \int_{\mathbb{R}^3}\left(\left(\phi\nabla \psi\right)(\bx+\mathbf{y},t)\times\nabla\right)\times \left[\chi\left(\frac{|\mathbf{y}|}{\rho}\right)\left(1-\chi\left(\frac{|\mathbf{y}|}{\delta}\right)\right) \frac{\mathbf{y}}{|\mathbf{y}|^3}\right] d\mathbf{y}\right|\nonumber\\
&\le& \frac{1}{4\pi}\left| \int_{\mathbb{R}^3}\left(1-\chi\left(\frac{|\mathbf{y}|}{\delta}\right)\right)
\left[\left(\left(\phi\nabla \psi\right)(\bx+\mathbf{y},t)\times
\nabla\chi\left(\frac{|\mathbf{y}|}{\rho}\right)\right)\times
 \frac{\mathbf{y}}{|\mathbf{y}|^3}\right] d\mathbf{y}\right|\nonumber\\
&&+ \frac{1}{4\pi}\left| \int_{\mathbb{R}^3}\chi\left(\frac{|\mathbf{y}|}{\rho}\right)
\left[\left(\left(\phi\nabla \psi\right)(\bx+\mathbf{y},t)\times
\nabla\left(1-\chi\left(\frac{|\mathbf{y}|}{\delta}\right)\right)\right)\times
 \frac{\mathbf{y}}{|\mathbf{y}|^3}\right] d\mathbf{y}\right|\nonumber\\
&&+ \frac{1}{4\pi}\left| \int_{\mathbb{R}^3}\chi\left(\frac{|\mathbf{y}|}{\rho}\right)
\left(1-\chi\left(\frac{|\mathbf{y}|}{\delta}\right)\right)
\left[\left(\left(\phi\nabla \psi\right)(\bx+\mathbf{y},t)\times
\nabla\right)\times
 \frac{\mathbf{y}}{|\mathbf{y}|^3}\right] d\mathbf{y}\right|\nonumber\\
&\equiv & C+D+E,\hspace{1cm} \forall \bx\in L^t.
\end{eqnarray}
By a direct calculation, we get
\begin{eqnarray}
\label{I5-1}
  C &\le& C\max_{\bx\in \Omega_0(t)}\left|\phi\nabla \psi\right|, \\
\label{I5-2}
D&\le& C\max_{\bx\in \Omega_0(t)}\left|\phi\nabla \psi\right|, \\
\label{I5-3}
E&\le& C\log\left(\frac{\rho}{\delta}\right)\max_{\bx\in \Omega_0(t)}\left|\phi\nabla \psi\right|.
\end{eqnarray}
By taking $\D\delta=\min\left(\frac{1}{\Omega(t)},\frac{\rho}{2}\right)$ and using the 
assumption $|\nabla\psi|\le C$ and the fact that
$|\phi|$, $\|\bu\|_2$ are bounded, we prove that
\begin{eqnarray}
   \max_{\bx\in L^t}|\mathbf{u}(\bx,t)|\le C_u\log \Omega(t). 
\end{eqnarray}
for some constant $C_u>0$ as long as $\Omega(t)>e$.

\end{proof}

\vspace{0.1in}
\noindent
{\bf Acknowledgments}
The research was in part supported by the National Science Foundation
through the grant DMS-0908546.

\end{document}